\documentclass{amsart}
\usepackage{bbm}
\usepackage{mathabx}
\usepackage{tensor}
\usepackage{amssymb}
\usepackage{amsmath}
\usepackage{amsfonts}
\usepackage{amsthm}
\usepackage{stmaryrd}
\usepackage[all]{xy}
\usepackage{mathrsfs}
\usepackage{graphicx}
\usepackage{hyperref}
\usepackage{color}
\usepackage{tikz-cd}
\usepackage{tkz-euclide}
\usepackage{tikz}
\usetikzlibrary{matrix}
\usepackage{enumitem}
\usepackage{amsmath,amscd}
\usepackage{pifont}
\usetikzlibrary{arrows}
\usepackage[all]{xy}
\usepackage{graphicx}
\usepackage{placeins}
\usepackage{xspace}
\usetikzlibrary{calc,intersections,through,backgrounds}

\newtheorem{thm}{Theorem}[section]
\newtheorem{theorem}[thm]{Theorem}

\newtheorem{corollary}[thm]{Corollary}
\newtheorem{lemma}[thm]{Lemma}
\newtheorem{proposition}[thm]{Proposition}

\theoremstyle{definition}

\newtheorem{definition}[thm]{Definition}

\newtheorem{notation}[thm]{Notation}

\newtheorem{remark}[thm]{Remark}

\numberwithin{equation}{section}

\numberwithin{mytheorem}{subsection}

\numberwithin{mytheorem}{subsection}
\numberwithin{myconjecture}{subsection}
\numberwithin{mydefinition}{subsection}
\numberwithin{myremark}{subsection}
\numberwithin{mysituation}{subsection}
\numberwithin{myhypothesis}{subsection}
\numberwithin{myquestion}{subsection}
\numberwithin{mynotation}{subsection}
\numberwithin{myfact}{subsection}
\numberwithin{myexamples}{subsection}
\numberwithin{myexample}{subsection}
\numberwithin{myconstruction}{subsection}
\numberwithin{mycaution}{subsection}
\numberwithin{myproposition}{subsection}
\numberwithin{mylemma}{subsection}
\numberwithin{mycorollary}{subsection}

\def\FF{\mathbb{F}}

\def\PP{\mathbb{P}}

\def\ZZ{\mathbb{Z}}

\def\calC{\mathcal{C}}

\def\calE{\mathcal{E}}

\def\calU{\mathcal{U}}
\def\calV{\mathcal{V}}
\def\calW{\mathcal{W}}

\def\CH{\mathrm{CH}}

\def\Fp{\FF_p}

\def\P{\mathrm{Ch}}

\def\Zp){\ZZ_p}

\pgfkeys{tikz/mymatrixenv/.style={decoration=brace,every left delimiter/.style={xshift=3pt},every right delimiter/.style={xshift=-3pt}}}
\pgfkeys{tikz/mymatrix/.style={matrix of math nodes,left delimiter=[,right delimiter={]},inner sep=2pt,column sep=1em,row sep=0.5em,nodes={inner sep=0pt}}}
\pgfkeys{tikz/mymatrixbrace/.style={decorate,thick}}

\begin{document}\large
\title{Iteration of Polynomials $AX^d+C$ Over Finite Fields
}

\author{Rufei Ren}

\address{University of Rochester, Department of
Mathematics,  Hylan Building, 140 Trustee Road, Rochester, NY 14627}
\email{rren2@ur.rochester.edu}

\date{\today}

\maketitle

\setcounter{tocdepth}{1}
\tableofcontents
\begin{abstract}
For a polynomial $f(X)=AX^d+C \in \Fp[X]$ with $A\neq 0$ and $d\geq 2$, we prove that 
if $d\;|\;p-1$ and $f^{\circ i}(0)\neq f^{\circ j}(0)$ for $0\leq i<j\leq N$, then  $\#f^{\circ N}(\Fp) \sim \frac{2p}{(d-1)N},$
where $f^{\circ N}$ is the $N$-th iteration of $f$.
\end{abstract}

\section{Introduction}

We fix a prime $p$.
For a polynomial $f \in\Fp[X]$ we denote its iterates $f^{\circ j}(X)$ by setting $f^{\circ 0}(X) = X$ and $f^{\circ (j+1)}(X) = f(f^{\circ j}(X))$. In this paper, we focus on the polynomials of the form $f(X)=AX^d+C$ with $A\neq 0$ and $d\geq 2$. Our goal is to give a non-trivial upper bound for $\#f^{\circ N}(\Fp)$. However, as mentioned in \cite{HB}, when $f(X)=X^3+1$ and $p\equiv 2\pmod 3$, $\#f^{\circ N}(\Fp)$ achieves the trivial bound $p$. Therefore, in order to give a non-trivial bound for $\#f^{\circ N}(\Fp)$, it is crucial to restrict $p$ in some certain residue class modulo $d$. More precisely, we obtain the following theorem. 
\begin{theorem}\label{main thm}
	Let $f(X) = AX^d+C\in \Fp[X]$ with $A\neq 0$, $d\geq 2$ and  $d\;|\;p-1$. Suppose that 
	\begin{equation}\label{precondtion}
	f^{\circ i}(0)\neq f^{\circ j}(0)\quad \textrm{for}\quad 0\leq i<j\leq N.
	\end{equation}
	Then there exists  an absolute constant $M$  (not depending on $d, p, A, C$) such that whenever (1.1),
	we have \begin{equation}\label{goal}
	\left|\#f^{\circ N}(\Fp)- \mu_N\cdot p\right|\leq  Md^{d^{6N}}\sqrt{p},
	\end{equation} 
	where $\mu_N $ is defined recursively by taking
	$\mu_0= 1$ and	$$d\mu_r = 1 -(1-\mu_{r-1})^d.$$
	Moreover, we have 
	$$\mu_r\sim \frac{2}{(d-1)r}\quad \textrm{when} \quad r\to \infty.$$
\end{theorem} 

\begin{remark}
(1)	Note that \cite[Theorem 1]{HB} is the special case of Theorem~\ref{main thm} when $d=2$ and $p\neq 2$.
	
(2)	The condition that $d\;|\;p-1$ is essential for us to obtain the estimation of $\#f^{\circ N}(\Fp)$ as in Theorem~\ref{main thm}. It is simply because without it Lemma~\ref{first lemma} will fail even though a similar $P_i$ can be defined.
	We do not know if a similar result will hold for the polynomial $Ax^d+C$ when $d\nmid p-1$. We are certainly interested in this generalization. 
\end{remark}

Theorem~\ref{main thm} has the following corollaries, whose proofs together with  Theorem~\ref{main thm}'s will be given in \S~\ref{s3}. 

\begin{corollary}\label{cor1}
Let $f(X) =
AX^d + C \in \Fp[X] $ with $A\neq 0$, $d\geq 2$ and $d\;|\;p-1$. Then there exists some $D_d>0$ depending only on $d$ (in particular not depending on $A$ and $C$) such that 
	$f^{\circ i}(0)=f^{\circ j}(0)$ for some $i, j$ with
	$$i < j \leq 
	D_{d}\frac{p}{\log\log p}.$$
\end{corollary}

\begin{corollary}\label{cor2}
Let $\widetilde{f}=\widetilde{A}X^d +\widetilde{C} \in \ZZ[X]$ be an integer polynomial with $\widetilde{A}, \widetilde{C} > 0$.
For a prime $p$ we denote by
$\widetilde{f}_p(X)
$ the reduction of $\widetilde{f}$ in $\Fp[X]$.
Then there exists constant $p_{\widetilde A, \widetilde C, d}$ such that for all primes $p \geq p_{\widetilde A, \widetilde C,d}$, the sum of the cycle lengths (resp. the lengths of pre-cyclic paths) of $\Gamma_{\widetilde{f}_p}$ is at most $ \frac{21p\log d}{\log \log p}$ (resp. $ \frac{28p\log d}{\log \log p}$),
where  $\Gamma_{\widetilde{f}_p}$ is the directed graph
whose vertices are the elements of $\FF_p$ and such that there is an arrow from $P$ to $\widetilde{f}_p(P)$
for each $P\in\FF_p$.
\end{corollary}

The general line of the argument in this paper follows  the one in \cite{HB} closely, but the
fact that $f(x)-f(y)$ splits into $x-y$ and two or more factors (rather than just
one factor in the case of quadratics: $(ax^2+c)-(ay^2+c) = a(x-y)(x+y))$
makes the accounting more complicated and necessitates a more complicated
combinatorial argument. More precisely, we need to introduce the second function $\eta$ into a graph as in Definitions~\ref{graph} and \ref{proper}, which makes our estimation of the number of $(r,k,d)$-trees in \eqref{1} coarser than the one in \cite{HB}. Fortunately, this estimation still gives us an upper bound that we need.

Since this paper is essentially based on \cite{HB}, some lemmas would be similar to the ones in \cite{HB}. However, for the completeness, we will still give their proofs.

After submitting this paper, the author was told that Jamie Juul proves a stronger result of Theorem~\ref{main thm} in \cite{JJ} via Chebotarev’s Density theorem (a completely different method to the one used in the current paper) based on the paper \cite{JKMT} written by Par Kurlberg, Kalyani Madhu, Tom Tucker and herself.

\subsection*{Acknowledgment}

First the author would like to thank Professor Ambrus Pal for finding a referee and the referee for comments. 
The author also want to thank Daqing Wan, Tom Tucker and Shenhui Liu for their valuable discussions. 

\section{Decomposition of projective variety $\calC_N$ into union of absolutely irreducible curves.}
We assume that $d\;|\;p-1$ in the whole paper and fix a primitive $d$-th root of unity $\gamma\in \Fp$. We would drop the dependence of functions to $d$ when the context is clear.
\begin{notation}
	We put $P_i(X,Y)=X-\gamma^i Y$ and $F^{\circ r}(X,Z)=Z^{d^r}f^{\circ r}(\frac{X}{Z})$.
\end{notation}

\begin{lemma}\label{first lemma}
	Under the assumptions of Theorem~\ref{precondtion}, for every $0\leq r\leq N-1$ and $1\leq i\leq d-1$ the polynomials of the form $P_i(f^{\circ r}(X),f^{\circ r}(Y))$ are absolutely irreducible over $\Fp$.
\end{lemma}
\begin{proof}
	It is enough to show that there is no non-zero solution to
	\begin{equation}\label{eq:8}
		\nabla\left(W^{D}P_i\left(f^{\circ r}(\frac{U}{W}),f^{\circ r}(\frac{V}{W})\right)\right)=
		 \underline 0,
	\end{equation}
	where $D=d^r.$ 
	Suppose that $(u,v,w)$ is a non-zero solution to	\eqref{eq:8}. Then we have 
	\begin{eqnarray}\label{eq1}
	&	w^{D-1}(f^{\circ r})'(\frac{u}{w})=w^{D-1}\prod\limits_{j=0}^{r-1}d(f^{\circ j}(\frac{u}{w}))^{d-1}=0,\\
	&	w^{D-1}\gamma^i (f^{\circ r})'(\frac{v}{w})=w^{D-1}\gamma^i\prod\limits_{j=0}^{r-1}d(f^{\circ j}(\frac{v}{w}))^{d-1}=0,\\
	&	-uw^{D-2}(f^{\circ r})'(\frac{u}{w})+vw^{D-2}\gamma^i(f^{\circ r})'(\frac{v}{w})+Dw^{D-1}P_i\big(f^{\circ r}(\frac{u}{w}),f^{\circ r}(\frac{v}{w})\big)=0.
	\end{eqnarray}

Assume that $w=0$. 
From the equalities above,  there are $0\leq s\leq r-1$ and $0\leq t \leq r-1$ such that $F^{\circ s}(u,w)=F^{\circ t}(v,w)=0.$ Since
$$F^{\circ s}(u,0) = A^{\frac{d^s-1}{d-1}}u^{d^s}\quad \textrm{and}\quad F^{\circ t}(v,0) = A^{\frac{d^t-1}{d-1}}v^{d^t},$$ we obtain that $u=v=w=0,$ which is excluded.
 

Now we assume $w\neq 0$. Then there exist $u, v \in  \overline \FF_p$ such that $$f^{\circ s}(u) = f^{\circ t}(v) = 0\textrm{~and~}f^{\circ r}(u)-\gamma^i f^{\circ r}(v) = 0.$$

If $s = t$, we have $$P_i(f^{\circ (r-s)}(0),f^{\circ (r-s)}(0))=(1-\gamma^i)f^{\circ (r-s)}(0)= 0,$$ which implies $f^{\circ (r-s)}(0) = 0 = f^{\circ 0}(0)$ with $1 \leq  r-s \leq N,$ a contradiction to our assumption~\eqref{precondtion}.

If $s\neq t$, we have $P_i(f^{\circ (r-s)}(0), f^{\circ (r-t)}(0))=0$, which implies $f^{\circ (r-s+1)}(0)=f^{\circ (r-t+1)}(0)$, again a contradiction to our assumption~\eqref{precondtion}. 

Therefore, we conclude that the polynomial $P_i(f^{\circ r}(X),f^{\circ r}(Y))$ is irreducible over the algebraic completion $\overline\FF_p$ of $\FF_p$ for every $r \leq N-1$.
\end{proof}
\begin{notation}\label{no1}
	We put \begin{equation*}
	\rho_r(m):=\#\{x\in \Fp \;|\;f^{\circ r}(x)=m\}\quad \textrm{and}\quad \calW(r,k):=\sum\limits_{m\in \Fp}  \rho_r(m)^k.
	\end{equation*}
\end{notation}
Note that $\calW(r,k)$ plays an important role on counting $\#f^{\circ N}(\Fp)$ as in the proof of Theorem~\ref{main thm}.

As in \cite{HB}, for every $k\geq 0$, the function $\calW(r,k)$ is the number of solutions to
\begin{equation}\label{fr}
f^{\circ r}(x_1) = \dots = f^{\circ r}(x_k)\textrm{~in~}\Fp^k. 
\end{equation}

For every $r\geq 0$ we define the projective variety
 \begin{equation}\label{C0}
 \calC_r: X_0^{d^r}f^{\circ r}\left(\frac{X_1}{X_0}\right) = \dots =X_0^{d^r}f^{\circ r}\left(\frac{X_k}{X_0}\right).
 \end{equation}
 
We put 
 $$\Phi(X,Y;\ell,h):=  \begin{cases}
P_h(f^{\circ \ell}(X),f^{\circ \ell}(Y)),& \textrm{if~}\ell\geq 0\ \textrm{and}\ 1\leq h\leq d-1,\\
X - Y,& \textrm{if~}\ell = -1\ \textrm{and}\ h=0.
\end{cases}$$

Clearly, we have $$f^{\circ r}(X)-f^{\circ r}(Y)=A^{r}(X-Y)\prod_{\ell=0}^{r-1}\prod_{h=1}^{d-1} \Phi(X,Y; \ell,h).$$

For every solution $(x_1,\dots,x_k)$ to \eqref{fr} and every pair of distinct indices $1 \leq i\neq j \leq k$, if $x_i=x_j$, we put $\ell(x_i,x_j):=-1$ and $h(x_i,x_j):=0$; otherwise, we put $\ell(x_i,x_j)$ to be the smallest integer $\ell \in \{0, 1,\dots , r - 1\} $ such that $$\Phi(x_i, x_j ; \ell, h) = 0\quad \textrm{for some~} 1\leq h\leq d-1,$$
and denote $h(x_i,x_j):=h$.

Since 
$$\Phi(x_i,x_j;\ell(x_i,x_j),h)=f^{\circ \ell(x_i,x_j)}(x_i)-\gamma^hf^{\circ \ell(x_i,x_j)}(x_j)=0$$
holds for a unique $1\leq h\leq d-1$, we know that $h(x_i,x_j)$ is well-defined.

By definitions of $\ell(\cdot,\cdot)$ and $h(\cdot,\cdot)$, we have  
\begin{equation}
\begin{cases}
\ell(x_i,x_j)=\ell(x_j,x_i)=-1 \textrm{~and~}  h(x_j,x_i)= h(x_i,x_j)=0, \textrm{~or}\\
\ell(x_i,x_j)=\ell(x_j,x_i)\geq 0 \textrm{~and~}  h(x_j,x_i)+ h(x_i,x_j)=d.
\end{cases}
\end{equation}

%
\begin{notation}
	For a graph $G$ with $k$ vertices, we denote by $\calV_G$ and $\calE_G$ the sets of $G$'s vertices and edges, respectively.
\end{notation}
	
\begin{definition}\label{graph}
	An \emph{$(r, k,d)$-graph} is a graph $G$  with $\calV_G=\{1,2,\dots,k\}$ and each of its edges $\overline{ab}$ is associated two functions $\xi_G$ and $\eta_G$ on the ordered pairs $(a, b)$ and $(b, a)$ such that 
	\begin{itemize}
		\item $Range(\xi_G)=\{-1,\dots, r\}$ and $Range(\eta_G)=\{0,1,\dots,d-1\}$.
			\item When $\xi_G(a,b)=-1$, we have $\eta_G(b,a)=0.$
			\item When $\xi_G(a,b)\geq 0$, we have $\eta_G(b,a)\in \{1,\dots,d-1\}.$
		\item $\xi_G(a,b)=\xi_G(b,a) \textrm{~and~} \eta_G(b,a)+\eta_G(a,b)\equiv 0\pmod d.$
	\end{itemize}

 If there exits at least one edge $\overline{ab}$ in $G$ such that $\xi_G(a,b)=r$, we call $G$ a \emph{strict $(r,k,d)$-graph}. If for every pair of vertices in $G$ there is an edge connecting them, we call $G$ a \emph{complete $(r, k,d)$-graph}. 
\end{definition}

\begin{definition}\label{proper}
	Let $G$ be an $(r, k, d)$-graph. We call $G$ \emph{proper} if for every distinct vertices $a, b$ and $c$ such that $\overline{ab}$, $\overline{ac}$ and $\overline{bc}$ all belong to $\calE_G$,  we have the following.
	\begin{enumerate}
		\item If $\xi_G(a, b) = \xi_G(b, c)=-1$, then $\xi_G(a,c)=-1$.
			\item If $\xi_G(a, b) < \xi_G(b, c)$, then $\xi_G(a, c)=\xi_G(b, c)$ and $\eta_G(a, c)=\eta_G(b, c)$.
		\item If $0\leq \xi_G(a, b) = \xi_G(b, c)$ and $\eta_G(a,b) +\eta_G(b,c)\neq d$, then $\xi_G(a,c)=\xi_G(a,b)$ and $\eta_G(a,c)\equiv \eta_G(a,b) + \eta_G(b,c)\pmod d$.
		\item If $0\leq \xi_G(a, b) = \xi_G(b, c)$ and $\eta_G(a,b) +\eta_G(b,c)= d$, then $\xi_G(a,c)<\xi_G(a,b) = \xi_G(b,c)$.
	\end{enumerate} 
\end{definition}

\begin{lemma}\label{pro1}
	Let $G_{\underline x}$ be the complete $(r-1, k,d)$-graph associated to a solution $ \underline x=(x_1,\dots,x_k)$ to \eqref{fr} with $\xi_{G_{\underline x}}(a,b):=\ell(x_a,x_b)$ and $\eta_{G_{\underline x}}(a,b):=h(x_a,x_b)$ for every $\overline{ab}\in \calE_{G_{\underline x}}.$ Then $G_{\underline x}$ is proper.
\end{lemma}
\begin{proof}
	One can check that $G_{\underline x}$ satisfies all the conditions in Definition~\ref{proper}.
\end{proof}

We list the following properties for an $(r,k,d)$-graph $G$.

\begin{lemma}\label{div}
	Let $G$ be a complete proper strict $(r,k,d)$-graph with $r\geq 0$. Then there is  a unique partition $\{A_i\}_{i=1}^t$ of $\calV_G$ such that 
	if $a\in A_i$ and $a'\in A_j$ are two arbitrary vertices of $G$, then we have
	\[\begin{cases}
		 \xi_G(a, a') < r & \textrm{if~} i= j;\\
		 \xi_G(a, a') = r & \textrm{if~} i\neq j.
	\end{cases}\]
	Moreover, this $t$ satisfies $2\leq t\leq d$.	
%
\end{lemma}
\begin{proof}
	Let $a_0$ be an arbitrary vertex of $G$. We put $$B_0:=\{b\;|\; b\in \calV_G\ \textrm{such that}\ \xi_G(a_0,b)<r\}\cup \{a_0\}$$
	and	
	$$B_j:=\{b\;|\; b\in \calV_G\ \textrm{such that}\ \xi_G(a_0, b)=r\ \textrm{and}\ \eta_G(a_0,b)=j\}$$
	for every $1\leq j\leq d-1$.
	
	Relabeling the non-empty sets among $\{B_j\;|\; 0\leq j\leq d-1\}$, we obtain $\{A_i\;|\; 1\leq i\leq t\}$.
	By Definition~\ref{proper}, we know that the partition $\{A_i\;|\; 1\leq i\leq t\}$ satisfies all the  properties that are required in this lemma, and it is independent of the choices of the starting vertex $a_0$.
\end{proof}


\begin{definition}\label{inductive definition}
	\noindent
\begin{enumerate}
	\item 	Let $G_0$ be a proper $(r,k,d)$-graph. Assume that there are three distinct vertices $a,$ $b$ and $ c$ in $G_0$ such that the edges $\overline{ab}$ and $\overline{bc}$ belong to $\calE_{G_0}$ but $\overline{ac}$ does not; and the functions  $\xi_{G_0}(\cdot,\cdot)$ and $\eta_{G_0}(\cdot,\cdot)$ satisfy one of the following.
	\begin{enumerate}
		\item[(i)] $\xi_{G_0}(a,b) = \xi_{G_0}(b,c) = -1$;
		\item[(ii)] $0\leq \xi_{G_0}(a,b) = \xi_{G_0}(b,c) \textrm{~and~}   \eta_{G_0}(a,b)+\eta_{G_0}(b,c)\neq d;$ 
		\item[(iii)] $ -1\leq \xi_{G_0}(a,b) < \xi_{G_0}(b,c).$
	\end{enumerate}

	We write $G$ for the $(r,k,d)$-graph generated from ${G_0}$ by adding an extra edge $\overline{ac}$ and putting
		\begin{itemize}
			\item[(i')]  $\xi_{G}(a,c): = -1$ and $\eta_{G}(a,c):=0$ for the case (i);
			\item[(ii')]  $\xi_{G}(a,c):= \xi_{G_0}(a,b) = \xi_{G_0}(b,c)$ and $\eta_{G}(a,c)$ to be the integer in $\{1,\dots,d-1\}$ which is congruent to $\eta_{G_0}(a,b)+\eta_{G_0}(b,c)$ modulo $d$ for the case  (ii);
			\item[(iii')] $\xi_{G}(a,c):= \xi_{G_0}(b,c)$ and $\eta_{G}(a,c):= \eta_{G_0}(b,c)$ for the case  (iii).
		\end{itemize}
If $G$ is also proper, then we say that ${G_0}$ \emph{generates} $G$. 

\item  More generally,  for two proper $(r,k,d)$-graphs $G_0$ and $G$ if there is a chain of proper $(r,k,d)$-graphs $G_0,G_1,\dots,G_s:=G$ such that for every $0\leq h\leq s-1$,	the graph $G_{h+1}$ is generated from $G_h$ by adding one edge as in (1), then we also say that $G_0$ \emph{generates} $G$.

	Moreover, if $G$ cannot generate a bigger proper $(r,k,d)$-graph by (1), we call $G$ a \emph{maximal extension} of $G_0$.
\end{enumerate}
	\end{definition}

%
%
\begin{definition}
		For every two $(r,k,d)$-graphs $G_0$ and $G$, if $\calE_{G_0}\subset\calE_{G}$ and
	for every edge $\overline{ab}\in \calE_{G_0}$ we have $\xi_{G_0}(a,b)=\xi_{G}(a,b)$ and $\eta_{G_0}(a,b)=\eta_{G}(a,b)$,
	then we call $G_0$ a \emph{subgraph} of $G$.
\end{definition}

\begin{lemma}\label{welldefined}
	Every subgraph $G_0$ of a  complete proper $(r,k,d)$-graph  $G$ is proper and has a unique extension.
\end{lemma}
\begin{proof}
	Since $G$ is proper, we know that $G_0$ is also proper.

	Let $G'$ be a maximal extension of $G_0$ with the chain of proper $(r,k,d)$-graphs 
	$G_0, G_1,$ $\dots, G_s:=G'$ as in Definition~\ref{inductive definition}(1). 
	Since $G_0$ is a subgraph of $G$, we can inductively prove that $G_h$ is a subgraph of $G$ for every $0\leq h\leq s$. Let $G''$ be an another maximal extension of $G_0$ and $h_0$ be the smallest index such that $G_{h_0}$ is not a subgraph of $G''$. Let $\overline{ab}$ be the edge that we add in $G_{h_0-1}$ to obtain $G_{h_0}$. Since $G''$ is also a subgraph of $G$, by Definition~\ref{inductive definition}(1), we can add $\overline{ab}$ into $G''$ as well, which leads a contradiction to $G''$ being a maximal extension of $G_0$. 
\end{proof}

\begin{definition}\label{def:tree}
	Let $G$ be a proper $(r, k,d)$-graph. 
\begin{enumerate}
	\item 	We call a chain of edges $\{\overline{a_i a_{i+1}}\}_{i=0}^{s-1}$, i.e. $a_i\neq a_j$ for every $i,j\in \{0,1,\dots,s\}$ such that $i\neq j$,  \emph{potentially complete} in $G$, if there  exists $0\leq u\leq s-1$ such that 
	\begin{enumerate}
		\item $\xi_G({a_0,a_1} )\leq\cdots\leq  \xi_G({a_{u},a_{u+1}})\geq \cdots \geq \xi_G({a_{s-1},a_{s}})$
	with no consecutive equalities in this chain of inequalities.
	\item If $\xi_G(a_{i-1}, a_{i})=\xi_G(a_i, a_{i+1})\geq 0$, then we have $\eta_G(a_{i-1}, a_{i})+\eta_G(a_{i}, a_{i+1})\not\equiv 0\pmod d.$
	\end{enumerate}
	
%
	
	\item We call $G$ an \emph{$(r, k,d)$-tree} if 
it contains no loop and for every two vertices $a$ and $b$ the unique chain connecting $a$ and $b$ is potentially complete in $G$.

\item If $G$ is an \emph{$(r, k,d)$-tree}, we denote by $\P_{G}(a,b)$ the unique chain in $G_0$ connecting the vertices $a$ and $b$. When $a=b$, we put $\P_{G_0}(a,a):=\{a\}$.
\end{enumerate}

%

\end{definition}

\begin{lemma}\label{le:1}
	Let $k\geq 2$, and $G$ be an $(r,k,d)$-tree. Assume that $\overline{aa'}\in \calE_G$ satisfies \begin{equation}\label{eq:9}
	\xi_{G}(a,a')=\max\{\xi_{G}(b,b')\;|\; \overline{bb'}\in \calE_G
	\}.
	\end{equation}
	Then for every vertex $a_0$ in $G$ if we put $\P_{G}(a_0,{a}):=\{\overline{a_ia_{i+1}}\}_{i=0}^{s}$, where $a_{s+1}=a$, then
	the sequence $\{\xi_{G}(a_i,a_{i+1})\}_{i=0}^{s}$ is non-decreasing. 
\end{lemma}
\begin{proof}
	\noindent\textbf{Case I.} 
	When $a_s=a'$. Since $G$ is an $(r,k,d)$-tree, we know that
	 $\P_{G}(b,{a})$  is potentially complete.
	 Combined with \eqref{eq:9}, this implies that $\{\xi_{G}(a_i,a_{i+1})\}_{i=0}^{s-1}$ is non-deceasing.

		\noindent\textbf{Case II.}  When $a_s\neq a'$. By Definition~\ref{def:tree}(2), we know that $\P_{G}(b,{a'})=\P_{G}(b,{a})\cup \overline{aa'}$.
Replacing $a$ in \textbf{Case I} by $a'$ and putting $a_{s+1}:=a'$, we know that $\{\xi_{G}(a_i,a_{i+1})\}_{i=0}^{s}$ is non-deceasing, which completes the proof of this case.
\end{proof}
\begin{lemma}\label{construction}
For every complete proper $(r,k,d)$-graph $G$ there exists an $(r,k,d)$-tree $G_0$ which generates $G$.
\end{lemma}
\begin{proof}
	When $k=1$ the result is trivial for every $r\geq -1$. 
	
	Now assume that it holds for every $m\leq k$ and $r\geq -1$. For $m=k+1$, without loss of generality, we assume that $G$ is a complete proper strict $(r,k+1,d)$-graph. 
	
	If $r=-1$, we choose an arbitrary vertex $a$ in $G$. Connecting $a$ to every other vertices in $G$, we obtain a proper $(-1,k+1,d)$-graph $G_0$. Clearly, $G_0$ is a $(-1,k+1,d)$-tree with the unique maximal extension $G$.
	
	Now assume $r\geq 0$. 
		By Lemma~\ref{div}, we obtain a partition $\{A_i\}_{i=1}^t$ of $\calV_G$ such that
		\begin{enumerate}
	\item[(i)]	$	|A_i|\leq k$ for every $1\leq i\leq t$.
			\item[(ii)] For every $a\in A_i$ and $b\in A_j$ we have
			\[\begin{cases}
				\xi_{G}(a, b) < r&\textrm{if~} i= j,\\
				\xi_{G}(a, b) = r& \textrm{if~} i\neq j.
			\end{cases}\]
		\end{enumerate} 
	
	By induction, for every $1\leq i\leq t$ we can construct an $(r-1,|A_i|,d)$-tree $G_{i,0}$ which generates the restriction of $ G$ on $A_i$. Now we determine a representative $a_i$ for each $A_i$ as follows.
	\begin{itemize}
		\item If $|A_i|=1$, we put $a_i$ to be the unique vertex in $A_i$.
		\item If $|A_i|\geq 2$, we put $a_i$ to be a vertex in $A_i$ such that 
		$$\xi_G(a_i,a_i')=\max\{\xi_G(a,b)\;|\; a, b\in A_i
		\}$$ for some other vertex ${a_i'}\in A_i$.
	\end{itemize}

	We denote by $G_0$ the subgraph of $G$ such that
	$$\calE_{G_0}=\bigcup_{i=1}^t \calE_{G_{i,0}}\cup \{\overline{a_1a_i}\;|\;2\leq i\leq t\}.$$
	
	Now we prove that $G_0$ is an $(r,k+1,d)$-tree. Since $G_0$ contains no loop, it is enough to prove that for every two vertices $a$ and $b$ the chain $\P_{G_0}(a,b)$ is potentially complete.
	 
	 Let $a\in A_i$ and $b\in A_j$ be two distinct vertices of $G$.
	 
	\noindent\textbf{Case I.} When $i=j$. From $\P_{G_{0}}(a,b)=\P_{G_{i,0}}(a,b)$, we know that $\P_{G_0}(a,b)$ is potentially complete.
	
\noindent\textbf{Case II.} When $i\neq j$ and one of $i$ and $j$ is equal to $1$. Without loss of generality, we assume $i=1$. By the construction of $G_0$, we know that $$\P_{G_0}(a,b)=\P_{G_{1,0}}(a,{a_1})\cup \overline{a_1a_j}\cup \P_{G_{j,0}}({a_j},{b}). $$

	By Lemma~\ref{le:1}, we know that $\xi_G$	is increasing along $\P_{G_{i,0}}(a,a_1)$ and decreasing along $\P_{G_{j,0}}({a_j},{b})$. Combined with (ii), this implies that $\P_{G_0}(a,b)$ is potentially complete.
	
\noindent\textbf{Case III.} When $i\neq j$, $i\neq 1$ and $j\neq 1$. From the construction of $G_0$, we know that $$\P_{G_0}(a,b)=\P_{G_{i,0}}(a,{a_i})\cup \overline{a_ia_1}\cup \overline{a_1a_j}\cup \P_{G_{j,0}}({a_j},{b}).$$ 
 From $\eta_G(a_1,a_i)\neq \eta_G(a_1,a_j)$, we have $$\eta_G(a_i,a_1)+ \eta_G(a_1,a_j)\equiv -\eta_G(a_1,a_i)+ \eta_G(a_1,a_j)\not\equiv 0\pmod d.$$
 
Similar to the argument in \textbf{Case II}, we show that $\P_{G_0}(a,b)$ is potentially complete.

Now we are  left to show that $G_0$ generates $G$. For every vertices $a$ and $b$, since 
$\P_{G_0}(a,b)$ is potentially complete, using Definition~\ref{inductive definition}(1) inductively on this chain, we generate a graph $G'$ from $G_0$ such that $G'$ is a subgraph of $G$ and $G'$ contains $\overline{ab}$. 

By Lemma~\ref{welldefined}, $G_0$ has the unique maximal extension. Since $a$ and $b$ are arbitrarily chosen, we know that its maximal extension is exactly $G$, which finishes the proof.
%
%
\end{proof}
 
\begin{notation}
We correspond a proper $(r,k,d)$-graph $G$ a projective variety
	\begin{equation}\label{eq:CG0}
	\calC_G:\Phi(X_a,X_b,X_0;\xi_G(a,b),\eta_G(a,b))=0\quad\textrm{for every~ } \overline{ab}\in \calE_G,
	\end{equation}
	where $$\Phi(X,Y,Z;\ell,h)=\begin{cases}
	X-Y,& \textrm{when} \ \ell=-1;\\
	Z^{d^\ell}\Phi(\frac{X}{Z},\frac{Y}{Z};\ell,h),& \textrm{when} \  \ell\geq 0.
	\end{cases}$$
\end{notation}

Note that \begin{equation}\label{eq:ab}
\Phi(X_a,X_b,X_0;\xi_G(a,b),\eta_G(a,b))=-\gamma^{\eta_G(a,b)}\Phi(X_b,X_a,X_0;\xi_G(b,a),\eta_G(b,a)).
\end{equation}
The variety $\calC_G$ is defined independent of the order of $a$ and $b$.

\begin{lemma}\label{lemma1}
	For every $r\geq 0$ if $G$ is a complete proper $(r-1, k,d)$-graph, then $\calC_G$ is a subvariety of $\calC_r$.
\end{lemma}
\begin{proof}
	Let $\underline x$ be an arbitrary point on the variety $\calC_G$. Since $G$ is complete, for every two vertices $a$ and $b$ we have 
	$\Phi(x_a,x_b,x_0;\xi_G(a,b),\eta_G(a,b))=0$. Combined with $\xi_G(a,b)\leq r-1$, this implies
	$x_0^{d^r}f^{\circ r}(x_a)=x_0^{d^r}f^{\circ r}(x_b),$ which completes the proof.
 \qedhere
\end{proof}
\begin{lemma}\label{dif}
For two complete proper  $(r, k,d)$-graphs $G_1$ and $G_2$ if $G_1\neq G_2$, then $\calC_{G_1}\neq  \calC_{G_2}$.
\end{lemma}
\begin{proof}
	Suppose this lemma is false. Then there exist  two complete proper  $(r, k,d)$-graphs $G_1$ and $G_2$ such that $G_1\neq G_2$ and $\calC_{G_1}=\calC_{G_2}$.
	
	Now we have the following two cases.
	
	\textbf{Case I}. There exists an edge $\overline{ab}\in \calE_{G_1} $ such that $\ell_1:=\xi_{G_1}(a,b)>\ell_2:=\xi_{G_2}(a,b)$. Consider the graph $G_2$. We have $f^{\circ \ell_2}(X_a)=\gamma^{\eta_{G_2}(a,b)}f^{\circ \ell_2}(X_b)$, which implies $f^{\circ \ell_1}(X_a)=f^{\circ \ell_1}(X_b)$. From $\calC_{G_1}=\calC_{G_2}$, we obtain $\xi_{G_1}(a,b)<\ell_1$,
	 a contradiction.
	
		\textbf{Case II}. There exists an edge $\overline{ab} \in \calE_{G_1}$ such that $$\xi_{G_1}(a,b)=\xi_{G_2}(a,b), \ \textrm{but}\ \eta_{G_1}(a,b)\neq \eta_{G_2}(a,b).$$
		
		Denote $\ell:=\xi_{G_1}(a,b),$ $h_1:=\eta_{G_1}(a,b)$ and  $h_2:=\eta_{G_2}(a,b).$
		 Then we have $f^{\circ \ell}(X_a)=\gamma^{h_1}f^{\circ \ell}(X_b)$ and  
		$f^{\circ \ell}(X_a)=\gamma^{h_2}f^{\circ \ell}(X_b)$, which implies $f^{\circ \ell}(X_a)=f^{\circ \ell}(X_b)=0$, a contradiction to $\xi_{G_1}(a,b)=\ell$.
\end{proof}

Recall that we define $\calW(r,k)$ in Notation~\ref{no1}.
\begin{lemma}\label{decomposition}
For every $r\geq 0$ we have $$\calW(r,k)+(p-1)\gcd(p-1,d^r)^{k-2}=\#\Big(\bigcup_{G} \calC_G(\Fp)\Big),$$
	where the sum runs over all complete proper  $(r-1, k,d)$-graphs.
\end{lemma}
\begin{proof}
	By Lemmas~\ref{pro1} and \ref{lemma1}, we have 
	\begin{multline*}
	\#\Big(\bigcup_{G} \calC_G(\Fp)\Big)=	\#\calC_r(\Fp)=\#\{\underline x\in \calC_r(\Fp)\;|\; x_0\neq0\}+\#\{\underline x\in \calC_r(\Fp)\;|\; x_0=0\}\\
		=\calW(r,k)+\{(0,x_1,\dots,x_k)\;|\; x_1^{d^r}=\cdots=x_k^{d^r}\}\\
		=\calW(r,k)+(p-1)\gcd(p-1,d^r)^{k-2}.
\qedhere
	\end{multline*}

\end{proof}
For a complete proper $(r-1, k,d)$-graph $G$, in order to estimate $ \#\calC_G(\Fp)$, we need the following key proposition.
\begin{proposition}\label{key proposition}
	With the assumption~\eqref{precondtion}, for every complete proper $(N-1,k,d)$-graph $G$, the variety $\calC_G$ is an absolutely irreducible curve over $\Fp$ with degree at most $d^{(k-1)(N-1)}$.
\end{proposition}

We will give its proof after several lemmas.

\begin{lemma}\label{complete and tree}
	Let $G$ be a complete proper  $(r,k,d)$-graph and $G_0$ be an $(r,k,d)$-tree which generates $G$. Then we have $G$ and $G_0$ correspond the same projective variety.
\end{lemma}
\begin{proof}
	It is enough to show that $$\Phi(X_a,X_b,X_0;\xi_G(a,b),\eta_G(a,b))=0\quad \textrm{and}\quad \Phi(X_b,X_c,X_0;\xi_G(b,c),\eta_G(b,c))=0$$
	imply $\Phi(X_a,X_c,X_0;\xi_G(a,c),\eta_G(a,c))=0$ whenever $a$, $b$ and $c$ satisfy  one of the three cases in Definition~\ref{inductive definition}(1).
	
	For the case (i), we know that $X_a=X_b$ and $X_b=X_c$, which imply $X_a=X_c$.
	
	For the case (ii), we put $$\ell:=\xi_G(a,b)=\xi_G(b,c),\ h_1:=\eta_G(a,b) \textrm{~and~}  h_2:=\eta_G(b,c).$$ Then we have  
	$$f^{\circ \ell}(X_a)-\gamma^{h_1}f^{\circ \ell}(X_b)=0\quad\textrm{and}\quad f^{\circ \ell}(X_b)-\gamma^{h_2}f^{\circ \ell}(X_c)=0.$$
which imply \begin{equation}\label{2}
	f^{\circ \ell}(X_a)-\gamma^{h_1+h_2}f^{\circ \ell}(X_c)=0.
	\end{equation} Since $\gamma$ is a primitive $d$-th root of unity, the equality~\eqref{2} is exactly what $\xi_G(a,c)=\ell$ and $\eta_G(a,c)\equiv h_1+h_2\pmod d$ imply.
	
For the case (iii), we put $$\ell_1:=\xi_G(a,b),\ \ell_2:=\eta_G(b,c),\ h_1:=\eta_G(a,b)\textrm{~and~}h_2:=\eta_G(b,c).$$ Then we have \begin{equation}\label{11}
	f^{\circ \ell_1}(X_a)=\gamma^{h_1}f^{\circ \ell_1}(X_b)
	\end{equation} and 
	\begin{equation}\label{12}
	f^{\circ \ell_2}(X_b)-\gamma^{h_2}f^{\circ \ell_2}(X_c)=0.
	\end{equation}
	
	Consider the condition $\ell_1<\ell_2$ in (iii). We act $f^{\circ (\ell_2-\ell_1)}$ on the both sides of \eqref{11} and obtain $$f^{\circ \ell_2}(X_a)=f^{\circ \ell_2}(X_b).$$
	Combined with \eqref{12}, this implies $$f^{\circ \ell_2}(X_a)-\gamma^{h_2}f^{\circ \ell_2}(X_c)=0,$$ which is exactly the equality obtained from $\xi_G(a,c)=\ell_2$ and $\eta_G(a,c)=h_2$.
\end{proof}

%
%

\begin{lemma}\label{key emma}
	With the assumption~\eqref{precondtion}, the variety $\calC_{G_0}$ associated to an $(N-1,k,d)$-tree $G_0$ is a nonsingular complete intersection. Hence, $\calC_{G_0}$ is an absolutely irreducible curve over $\Fp$, with degree at most $d^{(k-1)(N-1)}$.
\end{lemma}
\begin{proof}
	To prove that $\calC_{G_0}$ is a nonsingular complete intersection we need to show that the vectors in the set $\{\nabla\Phi(x_a, x_b, x_0; \xi(a,b),\eta(a,b))\}_{\overline{ab}\in \calE_{G_0}}$ are linearly independent at every point $\underline x$ of $\calC_{G_0}$. Suppose to the contrary that
	\begin{equation}\label{linear}
	\sum_{\overline{ab}\in \calE_{G_0}}c_{ab}\nabla\Phi(x_a, x_b, x_0; \xi_{G_0}(a,b),\eta_{G_0}(a,b)) = \underline 0
	\end{equation}
	for some $\underline x\in \calC_{G_0}$ and some non-zero vector $\underline{c}\in \overline{\FF}_p^{d-1}.$

	We put $c_{ba}:=-\gamma^{\eta(a,b)}c_{ab}$. By \eqref{eq:ab}, we can freely swap $a$ and $b$ in \eqref{linear} without changing this equality.

	We put $G'$ to be the subgraph of $G_0$ consisting of $\overline{ab}\in \calC_{G_0}$ such that $c_{ab}\neq 0$. Let $\CH=\{\overline{a_ia_{i+1}}\;|\; 0\leq i\leq s-1, s\geq 1\}$ be an arbitrary maximal chain in $G'$. (Here ``maximal'' means that the chain cannot be extended further in $G'$, which is not necessary to be the longest). Clearly, $\CH$ satisfies that
		\begin{enumerate}
		\item $c_{a_ia_{i+1}}\neq 0$ for every $0\leq i\leq s-1$.
		\item There is no vertex $b\neq {a_1}$ such that $\overline{ba_0}\in \calE_{G_0}$ and $c_{a_0b}\neq 0$.
		\item There is no vertex ${b'}\neq {a_{s-1}}$ such that $\overline{b'a_s}\in \calE_{G_0}$ and $c_{a_sb'}\neq 0$.
	\end{enumerate}
		Moreover, since $G_0$ is an $(N-1,k,d)$-tree,   $\CH$ is potentially complete.

We put $L:=\max\limits_{0\leq i\leq s-1}\{\xi_{G_0}(a_i,a_{i+1}) \}$. 
From the property (2) of $\CH$, we have \begin{equation}\label{eq:2}
	(\partial/\partial x_{a_0})\Phi(x_{a_0}, x_{a_1}, x_0; \xi_{G_0}(a_0,a_1),\eta_{G_0}(a_0,a_1))=0,
	\end{equation}
	which forces \begin{equation}\label{xi}
	\xi_{G_0}(a_0,a_1)\geq 0,
	\end{equation} since otherwise we have 	
	\begin{multline*}
	(\partial/\partial x_{a_0}	)\left(\sum_{\overline{ab}\in \calE_{G_0}}c_{ab}\nabla\Phi(x_a, x_b, x_0; \xi_{G_0}(a,b),\eta_{G_0}(a,b))\right)\\
	=c_{a_0a_1}(\partial/\partial x_{a_0})\Phi(x_{a_0}, x_{a_1}, x_0; \xi_{G_0}(a_0,a_1),\eta_{G_0}(a_0,a_1))\\
	=c_{a_0a_1}(\partial/\partial x_{a_0})(x_{a_0}-x_{a_1})=c_{a_0a_1}\neq 0.
	\end{multline*}
	
	From \eqref{xi}, we can write \eqref{eq:2} explicitly as $$(Ad)^{\xi_{G_0}(a_0,a_1)}  \prod_{i=0}^{\xi_{G_0}(a_0,a_1)-1}(F^{\circ i}(x_{a_0},x_0))^{d-1}
	=0,$$
which implies  $F^{\circ j_0}(x_{a_0},x_0) = 0$ for some index $0 \leq j_0 \leq  \xi_{G_0}(a_0,a_1) - 1\leq L-1$.

Similarly, from the property (3) of $\CH$, we have $$\xi_{G_0}(a_{s-1},a_s)\geq 0,$$ which implies that there exists $0\leq j_s \leq  L-1$ such that $F^{\circ j_s}(x_{a_s},x_0) = 0$.

Since $\CH$  is potentially complete in $ G_0$, $\xi_{G_0}(a_0,a_{1})\geq 0$ and $\xi_{G_0}(a_{s-1},a_{s})\geq 0$ imply $\xi_{G_0}(a_i,a_{i+1})\geq 0$ for all $1\leq i\leq s-1$. 

	We next show that $x_0$ cannot vanish. If, on the contrary, we had $x_0 = 0$, then the relation $F^{\circ i}(x_{a_0},x_0) = 0$ would yield $x_{a_0} = 0.$ In general, if $x_{a_0} = x_0 = 0,$ then for any vertex ${a'}$ such that $\overline{a_0a'}\in \calE_{G_0}$, the
relation $$\Phi(x_{a_0},x_{a'}, x_0; \xi_{G_0}(a_0,a'), \eta_{G_0}(a_0,a')) = 0$$ implies $x_{a'}= 0.$  Since $G_0$ is connected, we have $x_a=0$ for all $a\in \calV_{G_0}$, which is impossible.

We may therefore assume that $x_0 = 1$, which takes us back to the affine situation, i.e.
\begin{equation}\label{eq:3}
 f^{\circ \xi_{G_0}(a_i,a_{i+1})}(x_{a_i})=\gamma^{\eta_{G_0}(a_i,a_{i+1})}f^{\circ \xi_{G_0}(a_i,a_{i+1})}(x_{a_{i+1}}) \textrm{~for every~}0\leq i\leq s-1,
\end{equation} 
and 
\begin{equation}\label{eq:1}
	f^{\circ j_0}(x_{a_0})=0 \textrm{~and~} f^{\circ j_s}(x_{a_s})=0 \textrm{~with~} 0\leq j_0\leq L-1\ \textrm{and}\ 0\leq  j_s \leq L-1.
\end{equation}

Based on the number of indices $0\leq i\leq s-1$ such that $ \xi_{G_0}(a_i,a_{i+1})=L$, we have the following two cases.

\noindent\textbf{Case I.} When there is a unique index $u$ in $\{0,\dots, s-1\}$ such that $ \xi_{G_0}(a_u,a_{u+1})=L$. Combined with \eqref{xi}, this shows that
for every $0\leq i\leq u-1$ we have $0\leq \xi_{G_0}(a_i,a_{i+1})\leq L-1$. Together with \eqref{eq:3}, this implies  $f^{\circ L}(x_{a_i})=f^{\circ L}(x_{a_{i+1}})$,
and hence 
\begin{equation}\label{eq:4}
f^{\circ L}(x_{a_0})=f^{\circ L}(x_{a_{u}}).
\end{equation}

Similarly, we have 
$f^{\circ L}(x_{a_{s}})=f^{\circ L}(x_{a_{u+1}}).$
Combining it with \eqref{eq:3} for $i=u$ and \eqref{eq:4}, we have
$$f^{\circ L}(x_{a_0})=\gamma^{\eta_{G_0}(a_u,a_{u+1})}f^{\circ L}(x_{a_{s}}).$$
Together with \eqref{eq:1}, this equality implies 
\begin{equation}\label{eq:5}
f^{\circ (L-j_0)}(0)=\gamma^{\eta_{G_0}(a_u,a_{u+1})}f^{\circ (L-j_s)}(0).
\end{equation}

If $j_0=j_s$, since  $\gamma^{\eta_{G_0}(a_u,a_{u+1})}\neq 1$, we have $f^{\circ (L-j_0)}(0)=0$. Combined with $L\leq N-1$, this leads to a contradiction to our assumption~\eqref{precondtion}.

Now assume $j_0\neq j_s$. Without loss of generality, we assume $j_0<j_s$. From \eqref{eq:5}, we have 
$f^{\circ (L-j_0+1)}(0)=f^{\circ (L-j_s+1)}(0),$
and hence $f^{\circ ( L+1)}(0)=f^{\circ (L-j_s+j_0+1)}(0),$
which contradicts our assumption~\eqref{precondtion}.

\noindent\textbf{Case II.} When there is an index $u$ in $\{0,\dots, s-2\}$ such that $$ \xi_{G_0}(a_u,a_{u+1})=\xi_{G_0}(a_{u+1},a_{u+2})=L,$$ and for every $i\notin\{u,u+1\}$ we have 
$0\leq \xi_{G_0}(a_i,a_{i+1})\leq L-1$.

Since $\left\{\overline{a_ia_{i+1}}\right\}_{i=0}^{s-1}$ is potentially complete in $ G_0$, we have $$\ell:=\eta_{G_0}(a_u,a_{u+1})+\eta_{G_0}(a_{u+1},a_{u+2})\not\equiv 0\pmod d.$$ 

From \eqref{eq:3}, we have 
\begin{equation*}
f^{\circ L}(x_{a_u})=\gamma^{\eta_{G_0}(a_u,a_{u+1})}f^{\circ L}(x_{a_{u+1}})\textrm{~and~} f^{\circ L}(x_{a_{u+1}})=\gamma^{\eta_{G_0}(a_{u+1},a_{u+2})}f^{\circ L}(x_{a_{u+2}}),
\end{equation*}
which implies
$$f^{\circ L}(x_{a_u})=\gamma^{\ell}f^{\circ L}(x_{a_{u+2}}).
$$

Similar to \textbf{Case I}, we have
\begin{equation}\label{eq:6}
f^{\circ (L-j_0)}(0)=\gamma^{\ell}f^{\circ (L-j_s)}(0).
\end{equation}

If $j_0=j_s$, since  $\gamma^{\ell}\neq 1$, we have $f^{\circ (L-j_0)}(0)=0$. Combined with $L\leq N-1$, this leads to a contradiction to our assumption~\eqref{precondtion}.

Now assume $j_0\neq j_s$. Without loss of generality, we assume $j_0< j_s$. From \eqref{eq:6}, we have 
$f^{\circ (L-j_0+1)}(0)=f^{\circ (L-j_s+1)}(0),$
and hence $f^{\circ ( L+1)}(0)=f^{\circ (L-j_s+j_0+1)}(0),$
which contradicts our assumption~\eqref{precondtion}.

Therefore, there is no non-trivial solution $\underline c$ to the system \eqref{linear}.

%
%
%
%
%
%
%
%
%
%
%

%
%

In general, a nonsingular complete intersection is necessarily absolutely irreducible,  with the codimension equal to the number of equations in the system and degree equal to the product of the degrees of the defining forms, see \cite[Lemma~3.2]{BB} for details. In our case, $\Phi(X_a, X_b, X_0; \xi_G(a,b),\eta_G(a,b))$ has degree at most $d^{N-1}$ and the system \eqref{eq:CG0} for $\calC_{G_0}$ has $k-1$ equations. Combining them, we complete the proof.
\end{proof}
\begin{proof}[Proof of the Proposition~\ref{key proposition}]
	It follows directly from Lemmas~\ref{construction}, \ref{complete and tree} and \ref{key emma}.
\end{proof}
\section{Counting Points and Counting Curves}\label{s3}
By Lemma~\ref{decomposition} and the inclusion-exclusion principle, we have
$$\sum_{G}\#\calC_G(\Fp)- \sum_{G_1\neq G_2}\#(\calC_{G_1} \cap \calC_{G_2})(\Fp)\leq \calW(r,k)+(p-1)\gcd(p-1,d^r)^{k-2}\leq \sum_{G}\#\calC_G(\Fp),$$
where $G$ runs over all distinct complete proper  $(r-1,k,d)$-graphs.

Let $\calU(r,k)$ be the number of distinct complete proper $(r,k,d)$-graphs.
 By convention, $\calU(r, 0) = 1$ for all $r\geq -1$.


Combining Lemma~\ref{dif} and Proposition~\ref{key proposition} with Bezout's Theorem, for every two distinct complete proper $(N-1,k,d)$-graphs $G_1$ and $G_2$, we have $$\#(\calC_{G_1} \cap \calC_{G_2})(\Fp)\leq d^{2kN}.$$

 Therefore, we have
 \begin{equation}\label{eq}
 \Big|\calW(N,k)+(p-1)\gcd(p-1,d^N)^{k-2}-\sum_{G} \#\calC_G(\Fp)\Big|\leq \calU(N-1,k)^2 d^{2kN}.
 \end{equation}

By Weil's ``Riemann
Hypothesis'', every absolutely irreducible projective curve $\calC$ defined over $\Fp$ satisfies
$$|\#\calC(\Fp)- (p + 1)| \leq 2g \sqrt p,$$
where $g$ is the genus of $ \calC$. In general, if $\calC$ is an irreducible non-degenerate curve of degree $D$ in $\PP_k$ (with $k \geq  2$), then according to the Castelnuovo genus bound \cite{GC}, one has
$$g \leq  (k-1)m(m-1)/2 + m\epsilon,$$
where $D-1=m(k-1)+\epsilon$ with $0\leq \epsilon<k-1$. This implies that $g \leq (D- 1)(D-2)/2$ irrespective of the degree of the ambient space in which $\calC$ lies. Hence, we have
\begin{equation}\label{eq2}
|\#\calC_G(\Fp) -(p+ 1)| \leq d^{2kN}\sqrt{p},
\end{equation}
since $\calC_G$ has degree at most $d^{kN}.$

Combining \eqref{eq} and \eqref{eq2}, we have 
\begin{multline}\label{est}
\Big|\calW(N, k)+(p-1)\gcd(p-1,d^N)^{k-2}- \calU (N-1, k)(p+ 1) \Big| \\ 
\leq \calU(N-1,k)^2 d^{2kN}+\calU(N-1,k)d^{2kN}\sqrt{p}.
\end{multline}
\begin{notation}
	For every $k\geq 1$ let $M_{k,t}$ be the set of partitions of $\{1,\dots,k\}$ of $t$ components, where $\{A_i\}$ and $\{B_i\}$ are treated as the same partition if there is a permutation $\sigma\in S_t$ such that $A_i=B_{\sigma(i)}$ for all $1\leq i\leq t$.
\end{notation}
\begin{definition}
	Recall that every complete proper strict $(r,k,d)$-graph $G$ with $r \geq 0$ corresponds a partition $\{A_i\}_{i=1}^{t}$ of $\{1, \dots , k\}$ as in Lemma~\ref{div} for some $1\leq t\leq d$. We call $G$ a \emph{$(\{A_i\},r)$-graph}, and denote the set by $M(\{A_i, r\})$.
\end{definition}

\begin{lemma}
	Let $k\geq 1$ and $1\leq t\leq d$. For every $\{A_i\}\in M_{k,t}$,  we have
	$$\#M(\{A_i\}, r)=\frac{(d-1)!}{(d-t)!}\prod_{i=1}^t\calU(r-1,|A_i|).$$
\end{lemma}
\begin{proof}
For every $1\leq i\leq t$ we choose an arbitrary vertex $a_i$ from $A_i$. We first determine $\eta(a_1,a_2)$, which can be chosen from the set $\{1,\dots,d-1\}$. 
	
	Since $a_2$ and $ a_3$ belong to different sets, we have $\eta(a_1,a_3)\neq \eta(a_1,a_2)$, which restricts $\eta(a_1,a_3)$ into a set of $d-2$ elements. We keep this iteration until determine $\eta(a_1,a_i)$ for all $2\leq i\leq t$. For each set $A_i$, there are $\calU(r-1,|A_i|)$ distinct complete proper $(r-1,|A_i|,d)$-graphs in total. Therefore, we obtain 
		\[\#M(\{A_i\},r)=\frac{(d-1)!}{(d-t)!}\prod_{i=1}^t\calU(r-1,|A_i|).\qedhere\]
\end{proof}

For every integer $r \geq  -1$ we define the power series
\begin{equation}\label{eq:E}
E(X;r):= \sum_{k=0}^\infty \frac{\calU(r,k)}{k!}X^k.
\end{equation}

Now we estimate $\calU(r,k)$.
By Lemmas~\ref{welldefined} and \ref{construction}, we know that   $\calU(r,k)$ can be bounded above by the number of $(r,k,d)$-trees. Therefore, it is enough to estimate the number of $(r,k,d)$-trees.

We first determine the edges of the trees. 
We connect $k-1$ pairs of vertices in a $k$-vertex graph and obtain $\binom{\frac{(k-1)k}{2}}{k-1}$ distinct graphs. Clearly, every $(r,k,d)$-tree has to coincide one of these graphs. 

On the other hand, for every edge of an  $(r,k,d)$-tree, say $\overline{ab}$, we have
$$-1\leq \xi(a,b)\leq r\quad \textrm{and}\quad 0\leq \eta(a,b)\leq d-1.$$

Therefore, by Stirling's formula, we get a bound for $\calU(r, k)$ as
\begin{equation}\label{1}
\calU(r, k)\leq \binom{\frac{(k-1)k}{2}}{k-1}(r+2)^{k-1}d^{k-1}\leq \frac{((r+2)dk^2)^{k-1}}{(k-1)!}
\leq C_0((r+2)dek)^{k}
\end{equation}
for some constant $C_0>0$, where $e$ is the base of the natural logarithms. Therefore, the power series $E(X;r)$ has radius of convergence $\frac{1}{(r+2)de^2}$ for every $r\geq -1$.

Combining \eqref{est} and \eqref{1}, we obtain 
\begin{equation}\label{bound}
\Big|\calW(N, k)+(p-1)\gcd(p-1,d^N)^{k-2}- \calU (N-1, k)(p+ 1) \Big|=O\Big(((N+2)dek)^{2k}d^{2kN}\sqrt{p}\Big).
\end{equation}

\begin{notation}
	For a partition $\{A_i\}_{i=1}^t$, we define a counting function $$S(\{A_i\})=s_1!s_2!\dots s_k!,$$
	where $s_n$ represents the number of $A_i$ in $\{A_i\}$ of cardinality $n$, i.e. $$s_n=\#\{1\leq i\leq t\;|\; |A_i|=n \}.$$
\end{notation}

We define the following equivalence relation on the set $M_{k,t}$ of partitions:
$\{A_i\}\sim \{B_i\}$ if the multisets $$\{|A_i|\;|\; 1\leq i\leq t\}=\{|B_i|\;|\; 1\leq i\leq t\}.$$

\begin{lemma}
	For each $r\geq 0$, we have 
	\begin{equation}\label{E}
	E(X;r)=\frac{(E(X;r-1))^d+d-1}{d}.
	\end{equation}
\end{lemma}
\begin{proof}
	For every partition $\{A_i\}\in M_{k,t}$ there are $$\frac{k!}{S(\{A_i\})\prod_{i=1}^t |A_i|! }$$ equivalent partitions to $\{A_i\}$.
By Lemma~\ref{div}, we have 
\begin{multline}\label{eq:a1}
\calU(r, k)-\calU(r-1,k)=\sum_{t=2}^d \sum_{\{A_i\}\in M_{k,t}}\#M(\{A_i\},r)\\
=\sum_{t=2}^d \sum_{\{A_i\}\in M_{k,t}/ \sim} \#\{\{B_i\}\;|\;\{B_i\}\sim \{A_i\}\}\#M(\{A_i\},r)\\
=\sum_{t=2}^d \sum_{\{A_i\}\in M_{k,t}/ \sim} \frac{k!}{S(\{A_i\})\prod_{i=1}^t |A_i|! }\frac{(d-1)!}{(d-t)!}\prod_{i=1}^t\calU(r-1,|A_i|).
\end{multline}	

Expanding $(E(X;r-1))^d$ gives us \begin{equation*}
(E(X;r-1))^d=1 + \sum_{k=1}^\infty \sum_{t=1}^{d}\sum_{\{A_i\}\in M_{k,t}/ \sim} \frac{d!}{S(\{A_i\})(d-t)!}\prod_{i=1}^{t} \frac{\calU(r-1,|A_i|)}{|A_i!|}X^k.
\end{equation*}
Combined with \eqref{eq:a1}, this implies
$$(E(X;r-1))^d=1+d\sum_{k=1}^\infty \Big(\frac{\calU(r-1,k)}{k!}+\frac{\calU(r, k)-\calU(r-1, k)}{k!}\Big)X^k
=dE(X;r)-d+1,$$
and hence \eqref{E}.
\end{proof}

Since $\calU (-1, k) = 1$ for all $k\geq 0$, we have $E(X; -1) = e^X$. By induction, we have
$$E(X; r) =\sum_{m=0}^{d^{r+1}}v(r,m) e^{mX}$$
with non-negative real coefficients $v(r, m)$ summing to $1$. We then see that
$$E(X; r) =\sum_{m=0}^{d^{r+1}}v(r,m) \sum_{k=0}^{\infty} \frac{(mX)^k}{k!}.$$

We clearly have absolute convergence for small $X$, and we rearrange
to get
$$E(X; r) =\sum_{k=0}^{\infty} \Big(\sum_{m=0}^{d^{r+1}}v(r,m)m^k\Big) \frac{X^k}{k!}.$$
Hence, we have
$$\calU(r,k) =  \sum_{m=0}^{d^{r+1}}v(r,m)m^k.$$

We also see that the coefficient $v(r, 0)$ satisfies the recurrence \begin{equation}\label{v}
v(r, 0)= \frac{d-1+v(r-1, 0)^d}{d}\textrm{~for every~} r\geq 0\end{equation}
with $v(-1, 0)=0.$ We can then check that $\mu_r =1-v(r-1, 0)$ has the initial value $\mu_0 = 1$ and satisfies the recurrence \begin{equation}\label{mu}
d\mu_r = 1 -(1-\mu_{r-1})^d
\end{equation}
described in Theorem~\ref{main thm}.

\begin{proof}[Proof of Theorem~\ref{main thm}]
Consider that \begin{equation}\label{eq:a2}
\#f^{\circ N}(\Fp)=p-\#\{m\in \Fp \;|\;\rho_N(m)=0\}.
\end{equation}
Since the equation $f^{\circ N}(X) = m$ has at most $d^N$ solutions, we will always have $0 \leq  \rho_N(m) \leq  d^N$, whence
\[\frac{1}{d^N!}\prod_{j=1}^{d^N}(j-\rho_N(m))=\begin{cases}
	1 & \rho_N(m)=0;\\
	0 & \rho_N(m)\neq 0.\\
\end{cases}\]
Setting
\begin{equation}\label{3}
Q(T):=\sum_{k=0}^{d^N}C_{N,k}T^k=\frac{1}{d^N!}\prod_{j=1}^{d^N}(j-T),
\end{equation}
we then have
\begin{equation}\label{4}
\begin{split}
\sum_{k=0}^{d^N}C_{N,k}\calW(N,k)=\sum_{k=0}^{d^N}\left(C_{N,k}\sum\limits_{m\in \Fp}  \rho_r(m)^k\right)=&\sum\limits_{m\in \Fp} Q(\rho_r(m))\\=&
\#\{m\in \Fp \;|\;\rho_N(m)=0\}.
\end{split}
\end{equation}

Our plan is to substitute the approximated value of $\calW(N, k)$ given by \eqref{est}.
We first investigate the contribution from the main term $$\calU(N-1, k)(p+ 1)-(p-1)\gcd(p-1,d^N)^{k-2}.$$ This produces
\begin{align*}
&(p+1)\sum_{k=0}^{d^N}C_{N,k}\calU(N-1,k)-(p-1)\sum_{k=0}^{d^N}C_{N,k}\gcd(p-1,d^N)^{k-2}\\
=&(p+1)\sum_{k=0}^{d^N}\left(C_{N,k} \sum_{m=0}^{d^{N}}v(N-1,m)m^k\right)-\frac{p-1}{\gcd(p-1,d^N)^2}\frac{1}{d^N!}\prod_{j=1}^{d^N}(j-\gcd(p-1,d^N))\\
=&(p+1) \sum_{m=0}^{d^N}\left(v(N-1,m)\sum_{k=0}^{d^N}C_{N,k}m^k\right).
\end{align*}

The identity \eqref{3} shows that this inner sum vanishes for $1 \leq m \leq d^N$, and takes the value $1$ for $ m = 0$. Thus, the main term for \eqref{4} is just $$(p+ 1)v(N-1, 0) = (p+ 1)(1 -\mu_N),$$ producing the leading term $\mu_N\cdot p$ in \eqref{goal} when combined with \eqref{eq:a2}.

Now we handle the contribution to \eqref{4} arising from the error term in \eqref{bound}. For every $N\geq 2$ it has an upper bound 
\begin{multline*}
\sum_{k=0}^{d^N}|C_{N,k}|((N+2)dek)^{2k}d^{2kN}\sqrt{p}\leq 
\sum_{k=0}^{d^N}|C_{N,k}|(2Nded^N)^{2k}d^{2kN}\sqrt{p}
\\\leq
\frac{1}{d^N!}\prod_{j=1}^{d^N}(j+(2Nd^{1+2N}e)^2)\sqrt{p}
\leq (d^N+4N^2d^{2+4N}e^2)^{d^N}\sqrt{p}
\leq d^{d^{6N}}\sqrt{p}.
\end{multline*}


Let $q_r: =\frac{1}{\mu_r}$ for every $r\geq 0$. 
We next prove \begin{equation}\label{aa}
q_r\geq \frac{(d-1)r}{2}+1
\end{equation} for every $r\geq 0$ inductively. 

When $r=0$, the equality \eqref{aa} follows directly from $\mu_0=1$.
Assume that \eqref{aa} holds for some $r\geq 0$. Now we prove that \eqref{aa} also holds for $r+1$.

Consider the polynomial \begin{equation}\label{eq:p}
P(x):=d(x+1)^d-\left(x+1+\frac{d-1}{2}\right)\left((x+1)^d-x^d\right).
\end{equation} We know that 	for every $1\leq k\leq d-1$ the coefficient of $x^k$ in \eqref{eq:p} is equal to 
\begin{align*}
\frac{d-1}{2}\binom{d}{k}-\binom{d}{k-1}=\binom{d}{k}\left(\frac{d-1}{2}-\frac{k}{d-k+1}\right)\geq 0
\end{align*} and the constant term of $P(x)$ is equal to $\frac{d-1}{2}>0$. 

Combining them, we have
\begin{equation*}	P(x)>0\textrm{ ~for all~} x\geq 0.
\end{equation*}

From $\mu_0= 1$ and \eqref{mu},  we have
$q_r\geq 1$
and
\begin{multline}\label{eq:7}
q_{r+1}=\frac{d}{1-(1-\frac{1}{q_{r}})^d}
=\frac{dq_r^d}{q_r^d-(q_r-1)^d}\\
=q_r+\frac{d-1}{2}+\frac{dq_r^d-(q_r+\frac{d-1}{2})(q_r^d-(q_r-1)^d)}{q_r^d-(q_r-1)^d}
\\=q_r+\frac{d-1}{2}+\frac{P(q_r-1)}{q_r^d-(q_r-1)^d}
\geq q_r+\frac{d-1}{2}.
\end{multline}

On the other hand, from $q_{r}\xrightarrow{r\to \infty}\infty$, we have $$\frac{dq_r^d-(q_r+\frac{d-1}{2})(q_r^d-(q_r-1)^d)}{q_r^d-(q_r-1)^d}<\frac{C_1}{q_r}\leq \frac{2C_1}{r(d-1)}$$
for some constant $C_1>0$.

Combined with \eqref{eq:7}, this implies \begin{equation*}
q_r\leq d+\frac{(r-1)(d-1)}{2}+\sum_{i=1}^{r-1}\frac{2C_1}{i(d-1)}
\leq d+\frac{(r-1)(d-1)}{2}+\frac{2C_1}{(d-1)} (1+\log(r-1)).
\end{equation*}
Therefore, we obtain $q_r\sim \frac{(d-1)r}{2}$, which completes the proof.
\end{proof}
\begin{proof}[Proof of Corollary~\ref{cor1}]
We will prove that there exists an integer $P$ (independent to $A$ and $C$) such that 
this corollary holds for all primes $p\geq P$. 
	For primes $p< P$, we can take $D_d$ large enough so that 	$$D_{d}\frac{p}{\log\log p} > p,$$ 
	then the statement is trivial. Therefore, it suffices to consider sufficiently large $p$'s.

	Taking $$N:=\left\lfloor\frac{\frac{\log\log p+\log \frac{1}{3}}{\log d}-1}{6}\right\rfloor$$ in \eqref{goal}, we have 
	$$6N\log d+\log \log d<(6N+1)\log d<\log\log p+\log\frac{1}{3},$$
	and hence $d^{d^{6N}}<p^{1/3}$. Then 
	 the error term $$O(d^{d^{6N}}\sqrt{p})= O\left(p^{5/6}\right)\ll \frac{p}{N}.$$
	Combined with Theorem~\ref{main thm}, this implies that one of the following cases has to happen:
	\begin{enumerate}
		\item  $f^{\circ i}(0) = f^{\circ j}(0)$
		for some $0\leq i < j \leq N\ll \frac{p}{N}$.
			\item $\#f^{\circ N}(\Fp) \leq \frac{2p}{(d-1)N}+\frac{p}{N}.$ 
	\end{enumerate}
	
	For the case (1), the corollary is trivial. Now we assume that $f$ satisfies (2). We put $k: = \left\lceil\frac{2p}{(d-1)N}+\frac{p}{N}\right\rceil$+1. Since $f^{\circ N}(0),$ $f^{\circ (N+1)}(0), \dots $, $f^{\circ (N+k)}(0)$ all belong to $f^{\circ N}(\Fp)$ and $f^{\circ N}(\Fp)$ has at most $k-1$ element, there exist distinct $i$, $j$ in $\{N+1,\dots,N+k\}$ such that $f^{\circ i}(0)=f^{\circ j}(0),$ which finishes the proof.
\end{proof}
\begin{proof}[Proof of Corollary~\ref{cor2}]
	By choosing $p_{\widetilde A, \widetilde C, d} >\widetilde A,$ we may assume that $p\;\nmid\; \widetilde A.$
%

With the assumption that $\widetilde{A},\widetilde{C}>0$, we know that the sequence
	$\widetilde{f}^0(0), \widetilde{f}^1(0), \widetilde{f}^2(0), \dots$ is strictly increasing with $\widetilde{f}^j(0) \leq (\widetilde A+\widetilde C)^{d^j-1}.$ Thus,
	if $p \geq (\widetilde A + \widetilde C)^{d^N}$, we cannot have $p \;|\; (\widetilde{f}^j(0) -\widetilde{f}^i(0))$ with $0 \leq i < j \leq N$. The
	assumptions in Theorem~\ref{main thm} will therefore hold when
	\begin{equation}\label{i}
	N \leq
	\frac{\log \log p}{\log d}-	\frac{\log \log (\widetilde{A}+\widetilde{C})}{\log d}.
	\end{equation}

We set \begin{equation}\label{v1}
N_0:=\lfloor\log\log p/(7\log {d})\rfloor+1.
\end{equation}
For $p$ large enough (relative to $\widetilde A, \widetilde C$ and $d$), $N_0$ satisfies \eqref{i}. 
By Theorem~\ref{main thm}, for $p$ large enough we have \begin{equation}\label{v2}
\#\widetilde f_p^{\circ N_0}(\Fp) <\left(\frac{2}{d-1}+1\right)p/N_0\leq \frac{3p}{N_0}< \frac{21p\log d}{\log \log p}.
\end{equation} 
Since all cycles lie inside the set $\widetilde f_p^{\circ N_0}(\Fp)$, we prove the first assertion of the corollary. Moreover, each pre-cycles has length less than or equal to $N_0+\#\widetilde f_p^{\circ N_0}(\Fp)$. Combining with \eqref{v1} and \eqref{v2},  for $p$ large enough we have
$$N_0+\#\widetilde f_p^{\circ N_0}(\Fp)< \frac{4p}{N_0}<\frac{28p\log d}{\log \log p},$$
which completes the proof of the second assertion.
\end{proof}

\end{document}